\documentclass[12pt]{article}
\usepackage{amsmath,amsthm,amssymb}
\usepackage{amssymb,latexsym}

\newtheorem{theorem}{Theorem}

\newtheorem{lemma}{Lemma}

\begin{document}

\title{\bf Diophantine approximation with mixed powers of Piatetski-Shapiro primes}

\author{\bf S. I. Dimitrov}

\date{}

\maketitle

\begin{abstract}
Let $[\,\cdot\,]$ denote the floor function. In this paper, we show that whenever $\eta$ is real and the constants $\lambda _i$ satisfy some necessary conditions,
then for any fixed $\frac{63}{64}<\gamma<1$ and $\theta>0$, there exist  infinitely many prime triples $p_1,\, p_2,\, p_3$ satisfying the inequality
\begin{equation*}
|\lambda _1p_1 + \lambda _2p_2 + \lambda _3p^2_3+\eta|<\big(\max \{p_1, p_2, p^2_3\}\big)^{\frac{63-64\gamma}{52}+\theta}
\end{equation*}
and such that $p_i=[n_i^{1/\gamma}]$, $i=1,\,2,\,3$.
\quad\\
{\bf  2020 Math.\ Subject Classification}: 11D75  $\cdot$  11P32
\end{abstract}

\section{Introduction and statement of the result}
\indent

The study of Diophantine inequalities involving prime numbers constitutes a rapidly evolving field within analytic number theory.
In 1967, A. Baker \cite{ABaker} proved that if  $\lambda_1,\lambda_2,\lambda_3$ are non-zero real numbers, not all of the same sign, $\lambda_1/\lambda_2$ is irrational, $\eta$ is real and $A>0$,
then there exist infinitely many prime triples $p_1,\,p_2,\,p_3$ such that
\begin{equation}\label{linear}
|\lambda_1p_1+\lambda_2p_2+\lambda_3p_3+\eta|<\varepsilon_1\,,
\end{equation}
where $\varepsilon_1=(\log \max p_j)^{-A}$.
Subsequently, the right-hand side of \eqref{linear} was improved by Ramachandra \cite{Ramachandra}, Vaughan \cite{Vaughan1974}, Lau and Liu \cite{Lau-Liu},  
Baker and Harman \cite{Baker1982} and Harman \cite{Harman1991}. 
The best result to date is due to Matom\"{a}ki \cite{Mato}, with $\varepsilon_1=( \max p_j)^{-\frac{2}{9}+\delta}$ and $\delta>0$.
In 2018, Gambini, Languasco and Zaccagnini \cite{Gambini} proved the existence of infinitely many triples of primes $p_1,\,p_2,\,p_3$ such that
\begin{equation}\label{square}
|\lambda_1p_1+\lambda_2p_2+\lambda_3p^2_3+\eta|<\varepsilon_2\,,
\end{equation}
where $\varepsilon_2=\big(\max \{p_1, p_2, p^2_3\}\big)^{-\frac{1}{12}+\delta}$ and $\delta>0$.
Weaker results were previously obtained in \cite{Li} and \cite{Mu}. 
Another interesting question is the study of Diophantine inequalities involving special prime numbers.
Let $P_l$ is a number with at most $l$ prime factors. 
Very recently  Todorova and Georgieva \cite{Todorova} solved inequality \eqref{square} with prime numbers $p_1,\, p_2,\, p_3$ such that $p_i+2=P_{l_i}$,\;$i=1,\,2,\,3$.
In 1953, Piatetski-Shapiro \cite{Shapiro1953} showed that for any fixed $\frac{11}{12}<\gamma<1$, there exist infinitely many prime numbers of the form $p = [n^{1/\gamma}]$.
Such primes are called Piatetski-Shapiro primes of type $\gamma$. 
Subsequently, the interval for $\gamma$ was sharpened many times and the best result to date has been supplied by Rivat and Wu \cite{Rivat-Wu} with $\frac{205}{243}<\gamma<1$.
The primes of the form $[n^{1/\gamma}]$ are very much in focus nowadays and many problems are solved using them. We mention, for example, papers \cite{Maier2022}, \cite{Maier2023} and \cite{Maier2025}. 
In 2022, the author \cite{Dimitrov2022} proved that for any fixed $\frac{37}{38}<\gamma<1$, the inequality \eqref{linear} is solvable with infinitely many Piatetski-Shapiro prime triples $p_1,\,p_2,\,p_3$ of type $\gamma$. 
As a continuation of these studies, we solve \eqref{square} with Piatetski-Shapiro primes. 
\begin{theorem}\label{Theorem}
Suppose that $\lambda_1,\lambda_2,\lambda_3$ are nonzero real numbers, not all of the same sign, that $\lambda_1/\lambda_2$ is irrational, and that $\eta$ is real. 
Let $\theta>0$ and $\gamma$ be fixed with $\frac{63}{64}<\gamma<1$.
Then there exist infinitely many ordered triples of Piatetski-Shapiro primes $p_1,\,p_2,\,p_3$ of type $\gamma$ such that
\begin{equation*}
|\lambda_1p_1+\lambda_2p_2+\lambda_3p^2_3+\eta|<\big(\max \{p_1, p_2, p^2_3\}\big)^{\frac{63-64\gamma}{52}+\theta}\,.
\end{equation*}
\end{theorem}

\section{Notations}
\indent

The letter $p$ will always denote a prime number. By $\delta$ we denote an arbitrarily small positive number, not the same in all appearances. 
As usual, $[t]$ and $\{t\}$ denote the integer part and the fractional part of $t$, respectively.
Moreover $\psi(t)=\{t\}-\frac{1}{2}$. We write  $e(t)=e^{2\pi it}$.
Let $\gamma$, $\theta$ and $\lambda_0$ be a real constants such that $\frac{63}{64}<\gamma<1$, $\theta>0$ and $0<\lambda_0<1$.
Since $\lambda_1/\lambda_2$ is irrational, there are infinitely many different convergents
$a_0/q_0$ to its continued fraction, with $\big|\frac{\lambda_1}{\lambda_2} - \frac{a_0}{q_0}\big|<\frac{1}{q_0^2}\,, (a_0, q_0) = 1\,, a_0\neq0$ and $q_0$ is arbitrary large.
Denote
\begin{align}
\label{X}
&X=q_0^\frac{13}{6}\,;\\
\label{Delta}
&\Delta=X^{-\frac{12}{13}}\log X\,;\\
\label{varepsilon}
&\varepsilon=X^{\frac{63-64\gamma}{52}+\theta}\,;\\
\label{H}
&H=\frac{\log^2X}{\varepsilon}\,;\\
\label{Sk}
&S_k(t)=\sum\limits_{\lambda_0X<p^k\leq X\atop{p=[n^{1/\gamma}]}}p^{1-\gamma}e(t p^k)\log p\,;\\
\label{Sigma}
&\Sigma_k(t)=\sum\limits_{\lambda_0X<p^k\leq X}e(t p^k)\log p\,;\\
\label{U}
&U_k(t)=\sum\limits_{\lambda_0X<n^k\leq X}e(t n^k)\,;
\end{align}

\begin{align}
\label{Omega}
&\Omega_k(t)=\sum\limits_{\lambda_0X<p^k\leq X}p^{1-\gamma}\big(\psi(-(p+1)^\gamma)-\psi(-p^\gamma)\big)e(t p^k)\log p\,;\\
\label{Ik}
&I_k(t)=\int\limits_{(\lambda_0X)^\frac{1}{k}}^{X^\frac{1}{k}}e(t y^k)\,dy\,.
\end{align}

\section{Preliminary lemmas}
\indent

\begin{lemma}\label{Fourier} Let $\varepsilon>0$ and $k\in \mathbb{N}$.
There exists a function $\theta(y)$ which is $k$ times continuously differentiable and
such that
\begin{align*}
&\theta(y)=1\hspace{12.5mm}\mbox{for }\hspace{5mm}|y|\leq 3\varepsilon/4\,;\\
&0<\theta(y)<1\hspace{5mm}\mbox{for}\hspace{7mm}3\varepsilon/4 <|y|< \varepsilon\,;\\
&\theta(y)=0\hspace{12.5mm}\mbox{for}\hspace{7mm}|y|\geq \varepsilon\,.
\end{align*}
and its Fourier transform
\begin{equation*}
\Theta(x)=\int\limits_{-\infty}^{\infty}\theta(y)e(-xy)dy
\end{equation*}
satisfies the inequality
\begin{equation*}
|\Theta(x)|\leq\min\bigg(\frac{7\varepsilon}{4},\frac{1}{\pi|x|},\frac{1}{\pi |x|}
\bigg(\frac{k}{2\pi |x|\varepsilon/8}\bigg)^k\bigg)\,.
\end{equation*}
\end{lemma}
\begin{proof}
See (\cite{Shapiro1952}).
\end{proof}

\begin{lemma}\label{Shapiroasymp} For any fixed $\frac{2426}{2817}<\gamma<1$, we have
\begin{equation*}
\sum\limits_{p\leq X\atop{p=[n^{1/\gamma}]}}1\sim \frac{X^\gamma}{\log X}\,.
\end{equation*}
\end{lemma}
\begin{proof}
See (\cite{Rivat-Sargos}, Theorem 1).
\end{proof}

\begin{lemma}\label{intSintI} We have
\begin{align*}
&\emph{(i)}\quad\quad\quad\int\limits_{-\Delta}^\Delta|S_1(t)|^2\,dt\ll X\log^3X\,,\quad\quad\int\limits_{0}^1|S_1(t)|^2\,dt\ll X^{2-\gamma}\log X\,,\\
&\emph{(ii)}\quad\quad\quad\int\limits_{-\Delta}^\Delta|I_1( t)|^2\,dt\ll X\,,\quad\quad \int\limits_{-\Delta}^\Delta|I_2( t)|^2\,dt\ll 1\,.
\end{align*}
\end{lemma}
\begin{proof} 
For (i) see (\cite{Dimitrov2022}, Lemma 6). For (ii) see (\cite{Todorova}, Lemma 15).
\end{proof}

\begin{lemma}\label{S1I1asymptotic} Let $|t|\leq\Delta$ and $\frac{11}{12}<\gamma<1$. Then the asymptotic formula
\begin{equation*}
S_1(t)=\gamma  I_1(t)+ \mathcal{O}\left(\frac{X}{e^{(\log X)^{1/5}}}\right)
\end{equation*}
holds.
\end{lemma}
\begin{proof}
See (\cite{Dimitrov2022}, Lemma 5).
\end{proof}

\begin{lemma}\label{Omegaest} Let $\frac{13}{14}<\gamma<1$. Then
\begin{equation*}
\Omega_2(t)\ll X^{\frac{21-7\gamma}{29}+\delta}\,.
\end{equation*}
\end{lemma}
\begin{proof}
See (\cite{Dimitrov2025}, Lemma 6).
\end{proof}

\begin{lemma}\label{Languasco} Let $k\geq1$ and $1/2X\leq Y\leq 1/2X^{1-\frac{5}{6k}+\delta}$. Then there exists a positive constant $c_1(\delta)$, which does not depend on $k$, such that
\begin{equation*}
\int\limits_{-Y}^Y\big|\Sigma_k(t)-U_k(t)\big|^2\,dt\ll\frac{X^{\frac{2}{k}-2}\log^2X}{Y}+Y^2X+X^{\frac{2}{k}-1}\mathrm{exp}\Bigg(-c_1\bigg(\frac{\log X}{\log\log X}\bigg)^{1/3}\Bigg)\,.
\end{equation*}
\end{lemma}
\begin{proof}
See (\cite{Gambini}, Lemma 1 and Lemma 2).
\end{proof}

\begin{lemma}\label{mathfrakSest} Let $\frac{11}{12}<\gamma<1$ and $\Delta\leq|t|\leq H$. Then there exists a sequence of real numbers $X_1,\,X_2,\ldots \to \infty $ such that
\begin{equation*}
\min\Big\{\big|S_1(\lambda_{1}t)\big|,\big|S_1(\lambda_2 t)\big|\Big\}\ll X_j^{\frac{37-12\gamma}{26}}\log^5X_j\,,\quad j=1,2,\dots\,.
\end{equation*}
\end{lemma}
\begin{proof}
See (\cite{Dimitrov2022}, Lemma 7).
\end{proof}

\begin{lemma}\label{intS4} We have
\begin{equation*}
\int\limits_{0}^1|S_2(t)|^4\,dt\ll X^{2-\gamma+\delta}\,.
\end{equation*}
\end{lemma}
\begin{proof}
See (\cite{Zhai}, (12)).
\end{proof}

\section{Beginning of the proof}
\indent

Consider the sum
\begin{equation}\label{Gamma}
\Gamma(X)=\sum\limits_{\lambda_0X<p_1,p_2,p^2_3\leq X\atop{p_i=[n^{1/\gamma}_i],\, i=1,2,3}}\theta(\lambda_1p_1+\lambda_2p_2+\lambda_3p^2_3+\eta)p^{1-\gamma}_1p^{1-\gamma}_2p^{1-\gamma}_3\log p_1\log p_2\log p_3\,.
\end{equation}
Using the inverse Fourier transform for the function $\theta(x)$, we obtain
\begin{align*}
\Gamma(X)&=\sum\limits_{\lambda_0X<p_1,p_2,p^2_3\leq X\atop{p_i=[n^{1/\gamma}_i],\, i=1,2,3}}p^{1-\gamma}_1p^{1-\gamma}_2p^{1-\gamma}_3\log p_1\log p_2\log p_3\\
&\times\int\limits_{-\infty}^{\infty}\Theta(t)e\big((\lambda_1p_1+\lambda_2p_2+\lambda_3p^2_3+\eta)t\big)\,dt\\
&=\int\limits_{-\infty}^{\infty}\Theta(t)S_1(\lambda_1t)S_1(\lambda_2t)S_2(\lambda_3t)e(\eta t)\,dt\,.
\end{align*}
We decompose $\Gamma(X)$ as follows
\begin{equation}\label{Gammadecomp}
\Gamma(X)=\Gamma_1(X)+\Gamma_2(X)+\Gamma_3(X)\,,
\end{equation}
where
\begin{align}
\label{Gamma1}
&\Gamma_1(X)=\int\limits_{|t|<\Delta}\Theta(t)S_1(\lambda_1t)S_1(\lambda_2t)S_2(\lambda_3t)e(\eta t)\,dt\,,\\
\label{Gamma2}
&\Gamma_2(X)=\int\limits_{\Delta\leq|t|\leq H}\Theta(t)S_1(\lambda_1t)S_1(\lambda_2t)S_2(\lambda_3t)e(\eta t)\,dt\,,\\
\label{Gamma3}
&\Gamma_3(X)=\int\limits_{|t|>H}\Theta(t)S_1(\lambda_1t)S_1(\lambda_1t)S_2(\lambda_3t)e(\eta t)\,dt\,.
\end{align}
We shall estimate $\Gamma_1(X),\,\Gamma_2(X)$ and $\Gamma_3(X)$, respectively,
in the Sections \ref{SectionGamma1}, \ref{SectionGamma2} and \ref{SectionGamma3}.
In Section \ref{Sectionfinal} we shall complete the proof of Theorem \ref{Theorem}.

\section{Lower bound for $\mathbf{\Gamma_1(X)}$}\label{SectionGamma1}
\indent

\begin{lemma}\label{S2asymptotic} Let $\frac{13}{14}<\gamma<1$. Then
\begin{equation*}
S_2(t)=\gamma\Sigma_2(t)+\mathcal{O}\left(X^{\frac{21-7\gamma}{29}+\delta}\right)\,.
\end{equation*}
\end{lemma}
\begin{proof}
From \eqref{Sk}, \eqref{Sigma}, \eqref{Omega} and the well-known asymptotic formula
\begin{equation*}
(p+1)^\gamma-p^\gamma=\gamma p^{\gamma-1}+\mathcal{O}\left(p^{\gamma-2}\right)
\end{equation*}
we write
\begin{align}\label{S2est1}
S_2(t)&=\sum\limits_{\lambda_0X<p^2\leq X}p^{1-\gamma}\big([-p^\gamma]-[-(p+1)^\gamma]\big)e(t p^2)\log p\nonumber\\
&=\sum\limits_{\lambda_0X<p^2\leq X}p^{1-\gamma}\big((p+1)^\gamma-p^\gamma\big)e(t p^2)\log p\nonumber\\
&+\sum\limits_{\lambda_0X<p^2\leq X}p^{1-\gamma}\big(\psi(-(p+1)^\gamma)-\psi(-p^\gamma)\big)e(t p^2)\log p\nonumber\\
&=\gamma\Sigma_2(t)+\Omega_2(t)+\mathcal{O}(1)\,.
\end{align}
Bearing in mind \eqref{S2est1} and Lemma \ref{Omegaest}, we establish the statement in the lemma.
\end{proof}
Put
\begin{equation}\label{JX}
J(X)=\gamma^3\int\limits_{|t|<\Delta}\Theta(t)I_1(\lambda_1t)I_1(\lambda_2t)I_2(\lambda_3t)e(\eta t)\,dt\,.
\end{equation}
Now \eqref{Delta}, \eqref{Sk}, \eqref{Sigma}, \eqref{U}, \eqref{Ik}, \eqref{Gamma1}, \eqref{JX}, 
Cauchy's inequality, Lemma \ref{Fourier}, Lemma \ref{intSintI}, Lemma \ref{S1I1asymptotic} and Lemma \ref{S2asymptotic} imply 
\begin{align*}
\Gamma_1(X)-J(X)
&=\gamma^2\int\limits_{|t|<\Delta}\Theta(t)\Big(S_1(\lambda_1t)-\gamma I_1(\lambda_1t)\Big)I_1(\lambda_2t)I_2(\lambda_3t)e(\eta t)\,dt\nonumber\\
&+\gamma\int\limits_{|t|<\Delta}\Theta(t)S_1(\lambda_1t)\Big(S_1(\lambda_2t)-\gamma I_1(\lambda_2t)\Big)I_2(\lambda_3t)e(\eta t)\,dt\nonumber\\
&+\int\limits_{|t|<\Delta}\Theta(t)S_1(\lambda_1t)S_1(\lambda_2t)\Big(S_2(\lambda_3t)-\gamma I_2(\lambda_3t)\Big)e(\eta t)\,dt\nonumber\\
&\ll\varepsilon \frac{X}{e^{(\log X)^{1/5}}}\left[\Bigg(\int\limits_{|t|<\Delta}\big|I_1(\lambda_2t)\big|^2\,dt\Bigg)^\frac{1}{2}\Bigg(\int\limits_{|t|<\Delta}\big|I_2(\lambda_3t)\big|^2\,dt\Bigg)^\frac{1}{2}\right.\nonumber\\
&\left.+\Bigg(\int\limits_{|t|<\Delta}\big|S_1(\lambda_1t)\big|^2\,dt\Bigg)^\frac{1}{2}\Bigg(\int\limits_{|t|<\Delta}\big|I_2(\lambda_3t)\big|^2\,dt\Bigg)^\frac{1}{2}\right]\nonumber\\
&+\varepsilon \int\limits_{|t|<\Delta}\big|S_1(\lambda_1t)\big|\big|S_1(\lambda_2t)\big|\Big|\gamma\Sigma_2(\lambda_3t)-\gamma I_2(\lambda_3t)+\mathcal{O}\left(X^{\frac{21-7\gamma}{29}+\delta}\right)\Big|\,dt\nonumber\\
\end{align*}

\begin{align}\label{Gamma1-JX}
&\ll\varepsilon\frac{X^\frac{3}{2}}{e^{(\log X)^{1/6}}}+\varepsilon \int\limits_{|t|<\Delta}\big|S_1(\lambda_1t)\big|\big|S_1(\lambda_2t)\big|\big|\Sigma_2(\lambda_3t)-U_2(\lambda_3t)\big|\,dt\nonumber\\
&+\varepsilon \int\limits_{|t|<\Delta}\big|S_1(\lambda_1t)\big|\big|S_1(\lambda_2t)\big|\big|U_2(\lambda_3t)-I_2(\lambda_3t)\big|\,dt\nonumber\\
&=\varepsilon\Bigg(\frac{X^\frac{3}{2}}{e^{(\log X)^{1/6}}}+J_1+J_2\Bigg)\,,
\end{align}
say. Using  Cauchy's inequality, Lemma \ref{Shapiroasymp}, Lemma \ref{intSintI} and Lemma \ref{Languasco}, we get
\begin{align}\label{J1est}
J_1&\ll X\Bigg(\int\limits_{-\Delta}^\Delta\big|S_1(\lambda_1t)\big|^2\,dt\Bigg)^\frac{1}{2}\Bigg(\int\limits_{-\Delta}^\Delta\big|\Sigma_2(\lambda_3t)-U_2(\lambda_3t)\big|^2\,dt\Bigg)^\frac{1}{2}\ll\frac{X^\frac{3}{2}}{e^{(\log X)^{1/6}}}\,.
\end{align}
By Euler's summation formula, we have
\begin{equation}\label{I-U}
I_2(t)-U_2(t)\ll1+|t|X\,.
\end{equation}
From \eqref{I-U}, Cauchy's inequality and Lemma \ref{intSintI}, we derive
\begin{align}\label{J2est}
J_2&\ll (1+\Delta X)\Bigg(\int\limits_{-\Delta}^\Delta\big|S_1(\lambda_1t)\big|^2\,dt\Bigg)^\frac{1}{2}\Bigg(\int\limits_{-\Delta}^\Delta\big|S_1(\lambda_2t)\big|^2\,dt\Bigg)^\frac{1}{2}\ll\Delta X^2\log^3X\,.
\end{align}
On the other hand for the integral defined by \eqref{JX}, we write
\begin{equation}\label{JXest}
J(X)=B(X)+\Phi\,,
\end{equation}
where
\begin{equation*}
B(X)=\gamma^3\int\limits_{-\infty}^{\infty}\Theta(t)I_1(\lambda_1t)I_1(\lambda_2t)I_2(\lambda_3t)e(\eta t)\,dt
\end{equation*}
and
\begin{equation}\label{Phi}
\Phi\ll\int\limits_{\Delta}^{\infty }|\Theta(t)||I_1(\lambda_1t)I_1(\lambda_2t)I_2(\lambda_3t)|\,dt\,.
\end{equation}
Arguing as in (\cite{Dimitrov2015}, Lemma 4), we deduce that if
\begin{equation*}
\lambda_0<\min\bigg(\dfrac{\lambda_1}{4|\lambda_3|}\,,\dfrac{\lambda_2}{4|\lambda_3|}\,,\dfrac{1}{16}\bigg)
\end{equation*}
then
\begin{equation}\label{BXest}
B(X)\gg\varepsilon X^\frac{3}{2}\,.
\end{equation}
By \eqref{Ik} and (\cite{Titchmarsh}, Lemma 4.2), we get
\begin{equation}\label{Ikest}
I_k(t)\ll X^{\frac{1}{k}-1}\min\Big(X,\, |t|^{-1}\Big)\,.
\end{equation}
Using \eqref{Phi}, \eqref{Ikest} and Lemma \ref{Fourier}, we obtain
\begin{equation}\label{Phiest}
\Phi\ll\frac{\varepsilon X^\frac{1}{2}}{\Delta}\,.
\end{equation}
Bearing in mind \eqref{Delta}, \eqref{Gamma1-JX}, \eqref{J1est}, \eqref{J2est}, \eqref{JXest}, \eqref{BXest} and  \eqref{Phiest}, we establish
\begin{equation}\label{Gamma1est}
\Gamma_1(X)\gg\varepsilon X^\frac{3}{2}\,.
\end{equation}

\section{Upper bound for $\mathbf{\Gamma_2(X)}$}\label{SectionGamma2}
\indent

Put
\begin{equation}\label{mathfrakS}
\mathfrak{S}(t,X)=\min\Big\{\big|S_1(\lambda_{1}t)\big|,\big|S_1(\lambda_2 t)\big|\Big\}\,.
\end{equation}
Taking into account \eqref{Gamma2}, \eqref{mathfrakS}, Lemma \ref{Fourier} and Lemma \ref{mathfrakSest}, we deduce
\begin{align}\label{Gamma2est1}
\Gamma_2(X_j)&\ll\varepsilon\int\limits_{\Delta\leq|t|\leq H}\mathfrak{S}(t, X_j)^\frac{1}{2}\big|S_1(\lambda_1 t)\big|^\frac{1}{2}\big|S_1(\lambda_2 t)\big|\big|S_2(\lambda_3 t)\big|\,dt\nonumber\\
&+\varepsilon\int\limits_{\Delta\leq|t|\leq H}\mathfrak{S}(t, X_j)^\frac{1}{2}\big|S_1(\lambda_1 t)\big|\big|S_1(\lambda_2 t)\big|^\frac{1}{2}\big|S_2(\lambda_3 t)\big|\,dt\nonumber\\
&\ll\varepsilon X_j^{\frac{37-12\gamma}{52}+\delta}\big(\Psi_1+\Psi_2\big)\,,
\end{align}
where
\begin{align}
\label{Psi1}
&\Psi_1=\int\limits_{\Delta}^H\big|S_1(\lambda_1 t)\big|^\frac{1}{2}\big|S_1(\lambda_2 t)\big|\big|S_2(\lambda_3 t)\big|\,dt\,,\\
&\Psi_2=\int\limits_{\Delta}^H\big|S_1(\lambda_1 t)\big|\big|S_1(\lambda_2 t)\big|^\frac{1}{2}\big|S_2(\lambda_3 t)\big|\,dt\,.\nonumber
\end{align}
We estimate only $\Psi_1$ and the estimation of $\Psi_2$ proceeds in the same way. From \eqref{Psi1} and Cauchy's inequality, we derive
\begin{equation}\label{Psi1est1}
\Psi_1\ll\Bigg(\int\limits_{\Delta}^H\big|S_1(\lambda_2t)\big|^2\,dt\Bigg)^\frac{1}{2}\Bigg(\int\limits_{\Delta}^H\big|S_1(\lambda_1t)\big|^2\,dt\Bigg)^\frac{1}{4}
\Bigg(\int\limits_{\Delta}^H\big|S_2(\lambda_3 t)\big|^4\,dt\Bigg)^\frac{1}{4}\,.
\end{equation}
Using Lemma \ref{intSintI} (i), we obtain
\begin{equation}\label{DeltaH2}
\int\limits_{\Delta}^H\big|S_1(\lambda_kt)\big|^2\,dt\ll HX_j^{2-\gamma}\log X_j\,, \quad k=1, 2\,.
\end{equation}
By Lemma \ref{intS4}, we find
\begin{equation}\label{DeltaH4}
\int\limits_{\Delta}^H\big|S_2(\lambda_3t)\big|^4\,dt\ll HX_j^{2-\gamma+\delta}\,.
\end{equation}
Now \eqref{Psi1est1} -- \eqref{DeltaH4} yield 
\begin{equation}\label{Psi1est2}
\Psi_1\ll H X_j^{2-\gamma+\delta}\,.
\end{equation}
Combining \eqref{varepsilon}, \eqref{H}, \eqref{Gamma2est1} and \eqref{Psi1est2}, we get
\begin{equation}\label{Gamma2est2}
\Gamma_2(X_j)\ll X_j^{\frac{37-12\gamma}{52}+\delta}X_j^{2-\gamma+\delta}=X_j^{\frac{141-64\gamma}{52}+\delta}\ll\frac{\varepsilon X_j^\frac{3}{2}}{\log X_j}\,.
\end{equation}

\section{Upper bound for $\mathbf{\Gamma_3(X)}$}\label{SectionGamma3}
\indent

By \eqref{Sk}, \eqref{Gamma3}, Lemma \ref{Fourier}  and Lemma \ref{Shapiroasymp}, it follows
\begin{equation}\label{Gamma3est1}
\Gamma_3(X)\ll X^3\int\limits_{H}^{\infty}\frac{1}{t}\bigg(\frac{k}{2\pi t\varepsilon/8}\bigg)^k \,dt=\frac{X^3}{k}\bigg(\frac{4k}{\pi\varepsilon H}\bigg)^k\,.
\end{equation}
Choosing $k=[\log X]$ from \eqref{H} and \eqref{Gamma3est1}, we deduce
\begin{equation}\label{Gamma3est}
\Gamma_3(X)\ll1\,.
\end{equation}

\section{Proof of the Theorem}\label{Sectionfinal}
\indent

Summarizing  \eqref{varepsilon}, \eqref{Gammadecomp}, \eqref{Gamma1est}, \eqref{Gamma2est2} and \eqref{Gamma3est}, we derive
\begin{equation*}
\Gamma(X_j)\gg\varepsilon X_j^\frac{3}{2}=X_j^{\frac{141-64\gamma}{52}+\theta}\,.
\end{equation*}
The last estimation implies
\begin{equation}\label{Lowerbound}
\Gamma(X_j) \rightarrow\infty \quad \mbox{ as } \quad X_j\rightarrow\infty\,.
\end{equation}
Bearing in mind  \eqref{Gamma} and \eqref{Lowerbound} we establish Theorem \ref{Theorem}.

\vskip30pt
\footnotesize
\begin{flushleft}
S. I. Dimitrov\\
\quad\\
Faculty of Applied Mathematics and Informatics\\
Technical University of Sofia \\
Blvd. St. Kliment Ohridski 8 \\
Sofia 1000, Bulgaria\\
e-mail: sdimitrov@tu-sofia.bg\\
\end{flushleft}

\begin{flushleft}
Department of Bioinformatics and Mathematical Modelling\\
Institute of Biophysics and Biomedical Engineering\\
Bulgarian Academy of Sciences\\
Acad. G. Bonchev Str. Bl. 105, Sofia 1113, Bulgaria \\
e-mail: xyzstoyan@gmail.com\\
\end{flushleft}


\begin{thebibliography}{}

\bibitem{ABaker} A. Baker, {\it On some Diophantine inequalities involving primes},
J. Reine Angew. Math., {\bf 228}, (1967), 166 -- 181.

\bibitem{Baker1982} R. Baker, G. Harman, {\it Diophantine approximation by prime numbers},
J. Lond. Math. Soc., {\bf 25}, (1982), 201 -- 215.

\bibitem{Dimitrov2015} S. I. Dimitrov, T. Todorova, {\it Diophantine approximation by prime numbers of a special form},
Annuaire Univ. Sofia, Fac. Math. Inform., {\bf102}, (2015), 71 -- 90.

\bibitem{Dimitrov2022} S. I. Dimitrov, {\it Diophantine approximation by Piatetski-Shapiro primes},
Indian J. Pure Appl. Math., \textbf{53}, 4, (2022), 875 -- 883.

\bibitem{Dimitrov2025} S. I. Dimitrov, M. D. Lazarova, {\it On the distribution of $\alpha p^2$ modulo one over primes of the form $[n^c]$},
Ramanujan J., \textbf{68}, 3, (2025), Art. 79.

\bibitem{Gambini} A. Gambini, A. Languasco, A. Zaccagnini, {\it A Diophantine approximation problem with two primes and one $k$-th power of a prime}, 
J. Number Theory, \textbf{188}, (2018), 210 -- 228.

\bibitem{Harman1991} G. Harman, {\it Diophantine approximation by prime numbers},
J. Lond. Math. Soc., \textbf{44}, (1991), 218 -- 226.

\bibitem{Lau-Liu} K. W. Lau, M. C. Liu, {\it Linear approximation by primes},
Bull. Austral. Math. Soc., \textbf{19}, (1978), 457 -- 466.

\bibitem{Li} W. Li, T. Wang, {\it Diophantine approximation with two primes and one square of prime},
Chinese Quart. J. Math., \textbf{27}, (2012), 417 -- 423.

\bibitem{Maier2022} H. Maier, M. Rassias, {\it The ternary Goldbach problem with a missing digit and other primes of special types}, 
Analysis at Large: Dedicated to the Life and Work of Jean Bourgain, Springer, (2022), 333 -- 362. 

\bibitem{Maier2023} H. Maier, M. Rassias, {\it The ternary Goldbach problem with two Piatetski-Shapiro primes and a prime with a missing digit},
Commun. Contemp. Math., \textbf{25}, (2023), 2150101.

\bibitem{Maier2025} H. Maier, M. Rassias,  {\it Prime avoidance property of $k$-th powers of Piatetski–Shapiro primes}, 
J. Théor. Nombres Bordeaux, \textbf{37}, (2025), 715 -- 725.

\bibitem{Mato} K. Matom\"{a}ki, {\it Diophantine approximation by primes},
Glasgow Math. J., {\bf 52}, (2010), 87 -- 106.

\bibitem{Mu} Q. Mu, Y. Qu, {\it A Diophantine inequality with prime variables and mixed power},
Acta Math. Sinica (Chin. Ser.), {\bf 58}, (2015), 491 -- 500.

\bibitem{Shapiro1952} I. I. Piatetski-Shapiro, {\it On a variant of the Waring-Goldbach problem},
Mat. Sb., {\bf30}, (1952), 105 -- 120, (in Russian).

\bibitem{Shapiro1953} I. I. Piatetski-Shapiro,
{\it On the distribution of prime numbers in sequences of the form $[f(n)]$},
Mat. Sb., {\bf 33}, (1953), 559 -- 566.

\bibitem{Ramachandra} K. Ramachandra,  {\it On the sums $\sum\lambda_jf_j(p_j)$},
J. Reine Angew. Math., {\bf262/263}, (1973), 158 -- 165.

\bibitem{Rivat-Sargos} J. Rivat, P. Sargos, {\it Nombres premiers de la forme $[n^c]$},
Canad. J. Math., \textbf{53}, (2001), 414 -- 433.

\bibitem{Rivat-Wu} J. Rivat, J. Wu, {\it Prime numbers of the form $[n^c]$},
Glasg. Math. J, {\bf 43}, (2001), 237 -- 254.

\bibitem{Titchmarsh}E. Titchmarsh, {\it The Theory of the Riemann Zeta-function} 
(revised by D. R. Heath-Brown), Clarendon Press, Oxford (1986).

\bibitem{Todorova} T. Todorova, A. Georgieva, {\it A Diophantine inequality involving mixed powers of primes with a specific type}, 
Mathematics, \textbf{13}, (2025), Art. 3065.

\bibitem{Vaughan1974} R. C. Vaughan, {\it Diophantine approximation by prime numbers I},
Proc. Lond. Math. Soc. {\bf28}, (1974), 373 -- 384.

\bibitem{Zhai} W. Zhai, {\it On the Waring–Goldbach problem in thin sets of primes}, 
Acta Math. Sinica (Chin. Ser.), {\bf41}, (1998), 595 -- 608.

\end{thebibliography}
\end{document}